\documentclass[a4paper,12pt]{amsart}
\usepackage[latin9]{inputenc}
\usepackage{a4}
\usepackage[all]{xy}
\usepackage{amsfonts}
\usepackage{amssymb}
\usepackage{amsmath}
\usepackage{amsthm}
\usepackage{changepage}
\usepackage{nomencl}
\usepackage{geometry}
\usepackage{tikz}
\usepackage{verbatim}
\usetikzlibrary{matrix}
\usepackage[toc,page]{appendix}
\begin{document}
\providecommand{\keywords}[1]{\textbf{\textit{Keywords: }} #1}
\newtheorem{thm}{Theorem}[section]
\newtheorem{lemma}[thm]{Lemma}
\newtheorem{prop}[thm]{Proposition}
\newtheorem{cor}[thm]{Corollary}
\theoremstyle{definition}
\newtheorem{defi}{Definition}[section]
\theoremstyle{remark}
\newtheorem{remark}{Remark}
\newtheorem{prob}{Problem}
\newtheorem{conjecture}{Conjecture}
\newtheorem{ques}{Question}

\newcommand{\cc}{{\mathbb{C}}}   
\newcommand{\ff}{{\mathbb{F}}}  
\newcommand{\nn}{{\mathbb{N}}}   
\newcommand{\qq}{{\mathbb{Q}}}  
\newcommand{\mQ}{{\mathbb{Q}}}  
\newcommand{\rr}{{\mathbb{R}}}   
\newcommand{\zz}{{\mathbb{Z}}}  
\newcommand{\K}{\mathbb{K}}
\newcommand{\fp}{\mathfrak{p}}
\newcommand{\fP}{\mathfrak{P}}
\newcommand{\ra}{\rightarrow}

\newcommand\gexp{{\operatorname{ge}}}
\newcommand\lcm{{\operatorname{lcm}}}
\newcommand\Aut{{\operatorname{Aut}}}
\newcommand\Gal{{\operatorname{Gal}}}
\newcommand\ord{{\operatorname{ord}}}

\title[Unramified extensions over low degree fields ]{Unramified extensions over low degree number fields}

\author{Joachim K\"onig}
\email{jkoenig@kaist.ac.kr}

\author{Danny Neftin}
\email{dneftin@technion.ac.il}

\author{Jack Sonn}
\email{sonn@math.technion.ac.il}

\address{Department of Mathematical Sciences, KAIST, 291 Daehak-ro, Yuseong-gu, Daejeon 34141, South Korea}

\address{Department of Mathematics, Technion, Israel Institute of Technology, Haifa 32000, Israel}

\address{Department of Mathematics, Technion, Israel Institute of Technology, Haifa 32000, Israel}

\keywords{Galois theory; unramified extensions; bounded ramification; specialization of Galois covers}

\begin{abstract}
For various nonsolvable groups $G$, we prove the existence of extensions of the rationals $\mQ$ with Galois group $G$ and inertia groups of order dividing $\gexp(G)$, 
where $\gexp(G)$ is the smallest exponent of a generating set for $G$. 
For these groups $G$, this gives the existence of number fields of degree $\gexp(G)$ with an unramified $G$-extension. 
The existence of such extensions over $\mQ$ for all finite groups would imply that, for every finite group $G$, there exists a quadratic number field admitting an unramified $G$-extension,  as was recently conjectured. 
We also provide further evidence for the existence of such extensions for all finite groups, by proving their existence when $\mQ$ is replaced with a function field $k(t)$ where $k$ is an ample field. 
\end{abstract}
\maketitle

\section{Introduction}
The (\'etale) fundamental group of (the spectrum of) the ring of integers of a number field is a central object in number theory, 
in particular since its abelianization, the class group of $K$, is closely related to many classical problems in number theory.
The distribution of ($p$-parts of) class groups over imaginary quadratic fields has been extensively studied, and is expected to be uniform by the Cohen--Lenstra heuristics. 
These heuristics were  recently generalized to pro-$p$ groups by Boston--Bush--Hajir \cite{BB}, to pro-odd groups by Boston--Wood \cite{BW} and to finite groups by Wood \cite{MWood}. 
Very little is known in the nonabelian case concerning the mere existence of such extensions, that is concerning: 
\begin{ques}\label{ques1} Does every finite group appear as a Galois group of an unramified extension of some quadratic number field?
\end{ques}

On the other hand, It is well known that every finite group $G$ appears as a Galois group of an unramified extension over some number field. 
This is obtained by realizing $G$ as a Galois group of a tame (tamely ramified) extension $L_0/K$ over some number field, and finding a (not necessarily Galois) number field $M$ which is disjoint from $L_0$ and satisfies:
``for every prime $\fP$ of $M$, the ramification index of $\fP$ over its restriction $\fp$ to $K$ is divisible by the ramification index of $\fp$ in $L_0$".
Abhyankar's lemma then implies that $L:=L_0M$ is an unramified extension of $M$ with $\Gal(L/M)\cong G$. 
Moreover, the resulting extension $L/M$ is {\textit{tamely defined}} over $K$, that is, there is a tame Galois extension $L_0/K$ such that  $L_0\otimes_K M\cong L$. This is a common method for generating unramified extensions, e.g.~used in \cite{Kondo}, \cite{KRS}, \cite{HM}, \cite{KKS2}, \cite{Ozaki}. In fact, since the Inverse Galois Problem is known for many groups $G$, we restrict our consideration to such groups and hence choose $K=\mQ$.
One benefit of this approach is that it gives infinitely many number fields $M$ admitting an unramified $G$-extension with $[M:\mQ]=d$, where $d$ is the lcm of all ramification indices in $L_0/\mQ$. 
In fact, a positive density of all number fields of the form $M=\mQ(\sqrt[d]{N})$ will automatically have this property. 

Since the inertia groups of $L_0/\mQ$ generate $G$, the degree of such a number field $M$ cannot be smaller than the exponent of a generating tuple $S$ for $G$. Running over all generating tuples $S$, we call the minimal exponent of a generating set, the {\textit{generator exponent}} of $G$ and denote:  $$\gexp(G):=\min_S lcm\{ord(x)\mid x\in S\},$$ where $S$ ranges over all generating subsets of $G$. 
Abhyankar's lemma then implies that if $M$ admits an unramified extension $L/M$ which is tamely defined over $\mQ$, then $[M:\mQ]\geq \gexp(G)$.

In fact, it is unknown whether the minimal degree $e(\mQ,G)$, of a number field $K$ admitting an unramified $G$-extension defined over $\mQ$ by a tame extension, equals $\gexp(G)$. 
%
\begin{ques}
\label{conj:e3}
Does $e(\mQ,G)$ equal $\gexp(G)$? That is, are there infinitely many number fields $K$ of degree $\gexp(G)$ admitting an unramified $G$-extension which is defined over $\mQ$ by a tame extension?
\end{ques}

Since $\gexp(G)=2$ holds for many groups including all nonabelian almost simple groups, an affirmative answer supports the predictions concerning Question \ref{ques1}. In fact, we note  that an affirmative answer for all finite groups implies 
an affirmative answer to Question \ref{ques1} for every $G$, by Remark \ref{rem:implication}. Moreover, for groups $G$ with $\gexp(G)=2$, Abhyankar's lemma in fact yields that a positive proportion of quadratic number fields admit an unramified $G$-extension.

Well known cases in which $e(\qq,G) = \gexp(G)$ include symmetric, abelian groups and dihedral groups, cf.\ Section \ref{sec:red}.
An affirmative answer follows for an extensive class of $p$-groups  from Kim \cite{KKS2}. Namely, \cite{KKS2} shows that every $p$-group $G$ can be realized as a tamely defined unramified extension over a number field $M$ whose degree is the exponent $\exp(G)$. Although the equality $\gexp(G)=\exp(G)$ does not always hold (e.g. for $G=C_p\wr C_p$), it holds for many $p$-groups such as the family of regular $p$-groups, and in particular  all $p$-groups with nilpotency class $c\leq p-1$, see Remark \ref{rem:p-groups}. 

In this paper we provide further support for  the equality $e(\qq,G)=\gexp(G)$. First, we consider nonsolvable groups:  Proposition \ref{ge2} proves the equality for the small order nonabelian almost simple groups: 
\begin{equation}\label{equ:list}A_5,PSL_2(7), PSL_2(11), M_{11}, PSL_3(3), PGL_2(7), PSp_4(3).2, PSp_6(2).\end{equation}
We note that the main contribution here is the criteria, described below, which gives this list of examples and can be applied to many other almost simple groups. We then show that this family of realizations and the previously known realizations are closed under finite direct products, see Remark \ref{rem:direct}, and wreath products with symmetric groups, see Proposition \ref{prop:wreath}, giving an affirmative answer for infinite families such as $((M_{11}^k \wr S_{n_1}) \wr S_{n_2} \wr \cdots ) \wr  S_{n_r}$ for arbitrary positive integers $k,r, n_1,\ldots,n_r$. 

We provide further support in view of the analogy with function fields. In similarity with $\mQ$,  the rational function field $k(x)$ has no $k$-regular unramified extensions, and hence we analogously consider the minimal degree $e(k(x),G)$  of an extension $M/k(x)$ admitting a $k$-regular unramified $G$-extension  $L/M$, which is defined over $k(x)$. Here, we can consider ample fields $k$, due to their better understood arithmetic, see \cite{BBF},\cite{Pop3} for surveys. These include $k=\mathbb C((x))$, which due to its resemblance with a finite field is also known as a quasi-finite field, and its function field $k(x)$ resembles a global field. For ample $k$, we give an affirmative answer to both questions for arbitrary finite groups $G$. Namely,  $e(k(x),G)=\gexp(G)$ and  $G$ is the Galois group of an unramified $k$-regular extension $L/M$ over a quadratic extension $M/k(x)$, see Section \ref{sec:ff}. We note that further support for an affirmative answer follows from \cite{MWood}, which gives the desired extensions of $k(x)$ when $k$ is a sufficiently large finite field and $\gexp(G)=2$. 
Moreover, such results are expected in cases where $\gexp(G)>2$ as well, in view of \cite{EVW}, cf.~Remark \ref{rem:Fq}. 

Our main tool for constructing extensions $L_0/\mQ$ with nonabelian almost simple Galois group and given inertia groups over $\mQ$ is specialization of regular extensions of $L/\mQ(t)$ via Beckmann's theorem. We note that with different applications in mind, approaches similar to ours are developed in \cite{BSS}, \cite{BiluGillibert}, \cite[Prop.\ 2.5 and 2.6]{PV}. 
Beckmann's theorem allows reading the inertia groups of specializations from geometric inertia groups, except at a finite set $S$ of ``bad" primes. It is then  possible to force the completion of $L_0/\mQ$ at primes in $S$ to be unramified, except in the so called ``universally ramified primes" $U(L/\mQ(t))$. 
Our main criteria,  assert under various conditions that $U(L/\mQ(t))$ is the empty set. Namely, Theorem \ref{coprime_in} gives completely group theoretic conditions on the inertia groups of $L/\mQ(t))$ under which $U(L/\mQ(t))=\emptyset$. These conditions are easily verified in order to obtain the last three groups in \eqref{equ:list}. Similar conditions (Remark \ref{rem:main_crit}) are used on known realizations to obtain $M_{11}, PSL_2(11)$ and $PSL_3(3)$. Finally, $A_5$ and $PSL_2(7)$ are obtained by testing known polynomials.  

{\textbf{Acknowledgement:}}
We thank Brandon Alberts,  Nigel Boston, and Kwang--Seob Kim for valuable background information.
This research was supported by the Israel Science Foundation (grant No. 577/15).

\section{A reduction and first examples}\label{sec:red} 

\subsection*{Reduction} Let $K$ be a number field and $G$ a finite group. 
Denote by $e(K,G)$ the smallest degree $[M:K]$ of an overfield $M$ which admits an unramified $G$-extension which is defined over $K$ by a tame $G$-extension. 

\begin{lemma}\label{lem:basic}
The number $e=e(K,G)$ coincides with the minimal number $e'=e'(K,G)$ such that $K$ admits a tamely ramified $G$-extension with all inertia groups of order dividing $e'$. In particular, $e(\mQ,G)\geq \gexp(G)$. 
\end{lemma}
This relies on the well known Abhyankar lemma:
\begin{lemma} \label{lem:abh}
Let $L_1,L_2/K$ be two extensions and $L_1L_2$ their compositum. 
Let $\fp_1,\fp_2$ be primes of $L_1,L_2$, resp., with the same restriction $\fp$ to $K$, and $\frak P$ a prime of $L_1L_2$ lying over both. 
Then $e(\frak P/\fp_2) = e(\fp_1/\fp)/\gcd(e(\fp_1/\fp),e(\fp_2/\fp))$. In particular, $e(\frak P/\fp_2)=1$ if $e(\fp_1/\fp)$ divides $e(\fp_2/\fp)$. 
\end{lemma}
\begin{proof}[Proof of Lemma \ref{lem:basic}] 
$e'\le e$: 
Let $L/K$ be a $G$-extension with inertia groups of order dividing $e'$ over the primes $\fp_1,\ldots,\fp_r$. Let $v_{\fp_i}$ denote the normalized discrete valuation with maximal ideal $\fp_i$, for $i=1,\ldots,r$.
For  $\alpha\in K$ with $v_{\fp_i}(\alpha)=1$ for $i=1,\ldots,r$, the extension $\tilde{K}=K(\sqrt[e']{\alpha})$ is of degree $e'$ and is totally ramified over $\fp_1,\ldots,\fp_r$ \cite[Proposition 3.4.2]{Weiss}. By further choosing $\alpha$ to be congruent to $1$ at primes whose Frobenius elements in $L/K$ generate $G$, we may assume that $\tilde K$ is disjoint from $L/K$. 
It now follows from Abhyankar's Lemma that $L\tilde{K}/\tilde{K}$ is  an unramified $G$-extension.

Conversely, assume $[M:K]=e$, and 
let $L/M$ be an unramified $G$-extension defined over $K$ by a tame $G$-extension $L_0/K$. The lcm of ramification indices in $L_0/K$ is at least $e'$ by definition. 
Let $\frak{p}$ be a prime of $K$ which ramifies in $L_0/K$ with ramification index $e_{\frak p}$. 
Since every prime $\frak{P}$ of $M$ lying over $\frak{p}$ is unramified in $L/M$, Abhyankar's lemma implies that $e_{\frak p}$ divides the ramification index $e(\frak{P}/\frak{p})$.
It follows that the degree $[M:K]$ is divisible by $e_{\frak p}$ for every prime $\frak p$ of $K$, and hence by $e'$, proving $e\ge e'$. 

In the case $K=\mQ$, the inertia groups of $L_0/\mQ$ generate $G$ and hence $$e(\mQ,G)=e'(\mQ,G)\geq \gexp(G). \qedhere$$
\end{proof}
\begin{remark}\label{rem:PID} Although Lemma \ref{lem:abh} is stated over number fields, it holds similarly when $K$ is replaced by a function field $k(x)$ over a field $k$ of characteristic $0$:
In this case, we define $e(k(x),G)$  similarly to be the minimal degree $[M:k(x)]$ of an overfield $M$ 
 admitting an unramified $G$-extension  defined over $k(x)$ by an extension which additionally required to be $k$-regular. 
 The invariant $e':=e'(k(x),G)$ is also defined similarly as the minimal $e'$ such that $k(x)$ admits a (tamely ramified)  $G$-extension with all inertia groups of order dividing $e'$, which is additionally required to be $k$-regular.
The proof applies similarly since Abhyankar's lemma applies to function fields  \cite[Theorem 3.9.1]{St} of characteristic $0$, 
and since  $k(x)$ has no nontrivial unramified $k$-regular extensions. 
\end{remark}

\subsection*{Basic examples} The most basic examples are those of symmetric,  abelian, dihedral, and alternating groups over $K=\qq$. It is known that there exist infinitely many degree-$n$ extensions of $\qq$ with (Galois group $S_n$ and) squarefree discriminant, i.e.\ all inertia groups are generated by transpositions (cf.\ \cite{Kedlaya}). Thus, Lemma \ref{lem:basic} implies that $e(\qq,S_n) = 2$ for all $n$. 

For finite abelian groups $G$, one has  $\gexp(G) = \exp(G)$. 
Since the ramification indices divide $\exp(G)$ in every tame $G$-extension, 
and cyclotomic fields give examples for such extensions, it follows that $e(\qq,G) = \gexp(G)$. 

Let $n\geq 3$, and let $D_n$ be the dihedral group of degree $n$ and order $2n$. It is well known that there exist infinitely many quadratic number fields with an unramified $C_n$-extension \cite{Yamamoto}, and each such extension yields a $D_n$-extension $L/\qq$ with all inertia groups of order $2$. Thus. Lemma \ref{lem:basic} gives $e(\qq,D_n)=\gexp(D_n)=2$.

Infinitely many tamely ramified extensions $L/\qq$  with alternating Galois group $A_n$ for $n\geq 3$ with all inertia groups generated by $3$-cycles are realized  in \cite{Keo} using Mestre's construction \cite{Mestre} and building on the specialization method in Section \ref{sec:general}. 
%
In particular this gives infinitely many cubic extensions $\tilde K/\qq$ such that $\tilde K$ possesses an unramified $A_n$-extension. 
Note that the best possible bound $e(\mQ, A_n) = 2$ is only known for small $n$. 
We note that previously, tamely ramified $A_n$-extensions which are unramified over given finite sets $S$ had only been given by Plans--Vila \cite{PV2}. 

\begin{remark}\label{rem:p-groups}
Let $p$ be a prime. By \cite{KKS2}, for every $p$-group there exists a field $M$ of degree at most $\exp(G)$ and an unramified $G$-extension $L/M$ which is defined by a tame $G$-extension of $\mQ$.   For regular $p$-groups $G$ \cite[Hauptsatz 10.5]{Hup}, the order of the product $g=\prod_{i=1}^rg_i$ of elements $g_i\in G$, $i=1,\ldots,r$, is at most the maximum of the orders of $g_i$, $i=1,\ldots,r$. Thus, $g^{\gexp(G)}=1$ for every $g\in G$. 
Note that every $p$-group of nilpotency class $<p$ is regular  \cite[Paragraph III.10]{Hup}. 
\end{remark}

\section{The function field case}\label{sec:ff}
\subsection{Function fields and specializations - Setup}
Let $K$ be any field of characteristic $0$ and let $E/K(t)$ be a $K$-regular $G$-extension.
The extension $E/K(t)$ has finitely many branch points $p_1,...,p_r\in \overline{K}\cup\{\infty\}$, and associated to each branch point $p_i$ is a unique conjugacy class $C_i$ of $G$, corresponding to the automorphism $(t-p_i)^{1/e_i}\mapsto \zeta (t-p_i)^{1/e_i}$ of the Laurent series field
$\overline{K}(((t-p_i)^{1/e_i}))$, where $e_i$ is minimal such that $E$ embeds into $\overline{K}(((t-p_i)^{1/e_i}))$, and $\zeta$ is a fixed primitive $e_i$-th root of unity.\footnote{If $p_i=\infty$, one should replace $t-p_i$ by $1/t$.} This $e_i$ is the {\textit{ramification index}} at $p_i$, and equals the order of elements in the class $C_i$. The class tuple $(C_1,...,C_r)$ is called the {\textit{ramification type}} of $L/K(t)$. 
Let $\mathcal{H}_r(G,C)$ denote the Hurwitz space of all $G$-covers with ramification type $C$, cf.~\cite{Voe} where it is denoted by $\mathcal H_r^{in}(C)$. 

For each $a\in K$ which is not a branch point of $E/K(t)$, there is a unique residue field extension of places of $E$ extending the place $t\mapsto a$ (independent of the choice of place, as all these residue fields are isomorphic, and Galois). We call this field the specialization of $E/K(t)$ at $t\mapsto a$, and denote it by $E_a/K$.
\subsection{Ample fields}
Additional evidence for the above heuristics comes from the function field case. 
There is no known example of a finite group $G$ where  $e(\qq,G)\neq e(k(x),G)$. Moreover, $e(k(x),G)=\gexp(G)$ for every ample field $k$, as shown in Proposition \ref{prop:ff} below. Recall that $k$ is called {\textit{ample}} if every geometrically integral variety over $k$ either has no $k$-point or a Zariski dense set of $k$-points.
\begin{prop}\label{prop:ff} Let $G$ be a nontrivial finite group and $k$ be an ample field. 
Then 
\begin{enumerate}
\item[a)] 
There exist infinitely many linearly disjoint $k$-regular $G$-extensions $L_0/k(x)$ with ramification of order dividing $\gexp(G)$. In particular $e(k(x),G)=\gexp(G)$. 
 \item[b)] there exists infinitely many  quadratic extensions $M/k(x)$ admitting an unramified $G$-extension $L/M$. 
 \end{enumerate}
\end{prop}
\label{lem:fctfield}
\begin{proof}
For part a) we construct extensions using the so called $(1/2)$-Riemann existence theorem:
Let $S=\{x_1,\ldots,x_r\}$ be a set of generators of $G$, with each of order dividing $\gexp(G)$, and with the following property: For every $x\in S$, all generators of $\langle x\rangle$ are contained in $S$. The tuple $T:=(x_1,x_1^{-1}, \ldots, x_r,x_r^{-1})$
then has product $1$ and generates $G$. 
Let $C=(C_1,C_1^{-1},\ldots, C_r^{-1})$ denote the corresponding ramification type, where $C_i$ is the conjugacy class of $x_i$.
The  $(1/2)$-Riemann Existence Theorem of Pop (see \cite[Theorem 4.3.3]{Har} and \cite[Proposition 1.1]{Pop96}) ensures that
the Hurwitz space $\mathcal H_{2r}(G,C)$, of all  $G$-covers of $\mathbb P^1$ with ramification type $C$, has a $k$-rational point.
Since $k$ is ample this implies that the set of $k$-rational points $\mathcal H_{2r}(G,C)(k)$ is Zariski dense. 
Every $k$-rational point of the Hurwitz space gives a $G$-extension $L_0/k(x)$ whose inertia groups are generated by elements in $C_i$,  $i\in \{1,\ldots,r\}$. 
Since the set of such points is dense, we may find infinitely many such extensions $L_0/k(x)$ with pairwise disjoint branch points, and hence linearly disjoint, proving the first assertion.
 By Remark \ref{rem:PID}, we have $e(k(x),G) = e'(k(x),G)\geq \gexp(G)$,  
and equality follows from the first assertion. 


For b), let $\Gamma := G\wr C_2 = (G\times G)\rtimes C_2$ be the standard wreath product, and let $\hat G$ be the subgroup of $\Gamma$ generated by the set $E$ of all elements of the form $(g,g^{-1})\cdot a$, with $g\in G$, and $a$ the generator of $C_2$. Since all elements of $E$ are involutions, the group $\hat G$ is by definition generated by involutions lying outside of the subgroup $G\times G\le \Gamma$. 
Applying Part a) to $\hat G$ with respect to the generating set $E$, then yields the existence of infinitely many linearly disjoint $k$-regular  $\hat G$-extensions $F/k(x)$ with inertia groups generated by involutions in $E$. Since every inertia group is of order $2$, and is mapped isomorphically under the natural projection $\hat G\ra \hat G/K\cong C_2$, the extension $F/F^K$  is unramified with Galois group $K$ over the quadratic field $M:=F^K$. Since $\hat G$ contains all elements of the form $(g,g^{-1})$, $g\in G$, the natural projection onto the first coordinate is onto, giving a subextension $L/M$ with Galois group $G$. Since the fields $F/k(x)$ are chosen linearly disjoint, so are the fields $M$. 
\end{proof}

\begin{remark}\label{rem:implication}The proof of part b) shows that an affirmative answer to Question \ref{conj:e3} for all finite groups implies that every finite group $G$ appears as a Galois group of an unramified extension over infinitely many quadratic number fields. Indeed, such a $G$-extension arises from a $\hat G$-extension with inertia groups in $E$, where  $\hat G\subset G\wr C_2$ and $E$ are as in the proof. 
\end{remark}

\begin{remark}\label{rem:Fq}
Proposition \ref{prop:ff}  is expected to hold when $k$ is replaced with a finite field $\mathbb F_q$ for a sufficiently large prime power $q$ if a similar description of connected components of Hurwitz spaces to \cite[Section 8]{EVW}  applies\footnote{Unfortunately, the corresponding result in \cite[Section 8.9]{EVW2009} remains to be proved.}.  If $q$ is sufficiently large, and the corresponding Hurwitz space has an $\mathbb F_q$-rational connected component, the Lang--Weil estimates would show that this Hurwitz space has an $\mathbb F_q$-rational point, giving rise to a $G$-Galois covering of $\mathbb P^1$ over $\mathbb{F}_q$ with the desired ramification. 
\end{remark}

\section{Specialization of regular Galois extensions}
\subsection{Ramification in specializations of regular Galois extensions}
Throughout this section let $K$ be a number field. 
The following criterion relates the ramification in regular Galois extensions to ramification in specializations.
See also Theorem I.10.10 in \cite{MM}, which is closer  to our version.

Let $\mathfrak{p}$ be a finite prime of $K$, and $a_0\in \overline{K}$ be $\mathfrak{p}$-integral with minimal polynomial $f\in K[X]$. Define $I_\mathfrak{p}(a,a_0)$ as the multiplicity of $\mathfrak{p}$ in the fractional ideal generated by $f(a)$. Obviously,  $I_\mathfrak{p}(a,a_0)\ne 0$ holds only for finitely many prime ideals $\mathfrak{p}$ of~$K$. 
%
\begin{thm}[Beckmann, \cite{Beckmann91}, Prop.\ 4.2]
\label{beckmann}
Let $N/K(t)$ be a $K$-regular $G$-extension.
Assume that all branch points of $N/K(t)$ are finite.\footnote{This is assumed without loss of generality via a linear transformation in $t$, although the theorem remains valid at infinity after a minor adaptation.}
Then with the exception of finitely many primes, depending only on $N/K(t)$, the following holds for every prime $\mathfrak{p}$ of $K$.\\
If $a\in K$ is not a branch point of $N/K(t)$ then the following condition is necessary for $\mathfrak{p}$ to be ramified in the specialization $N_a/K$:
 $$\nu_i:=I_{\mathfrak{p}}(a,a_i)>0 \text{ for some (automatically unique) branch point $a_i$.}$$
Indeed, the inertia group of a ramified prime extending $\mathfrak{p}$ in the specialization $N_a/K$ is then conjugate in $G$ to $\langle\tau^{\nu_i}\rangle$, where $\tau$ is a generator of an inertia subgroup over the branch point $t\mapsto a_i$ of $K(t)$.
\end{thm}

We adopt the following useful definition from \cite{BSS}.
\begin{defi}[Universally ramified prime]
Let $N/K(t)$ be a $K$-regular Galois extension. A finite prime of $K$ is called universally ramified for $N/K(t)$ if it ramifies in every specialization $N_a/K$, where $a\in K$ is not a branch point of $N/K(t)$. The set of universally ramified primes is denoted by $U(N/K(t))$.
\end{defi}

\subsection{Hilbert irreducibility and weak approximation}
The following is a well-known consequence of Hilbert's irreducibility theorem in combination with Krasner's lemma (cf.\ Prop.\ 2.1 in \cite{PV}):
\begin{prop}
\label{weak_appr}
Let $K$ be a number field, $N/K(t)$ a $K$-regular $G$-extension, and $S$ a finite set of primes of $K$. Then there exist infinitely many $a\in K$ such that the specialization $N_a/K$ fulfills the following:
\begin{itemize}
\item[i)] $\Gal(N_a/K) \cong G$.
\item[ii)] $N_a/K$ is unramified at all primes in $S_0:=S\setminus U(N/K(t))$.
\end{itemize}
\end{prop}
\begin{proof}
Let $p\in S_0$. By definition, there is then a specialization $N_{a_p}/K$ (with $a_p\in K$) whose completion at $p$ is an unramified extension $F_p/K_p$. By Krasner's lemma, all $a\in K$ which are sufficiently close $p$-adically to $a_p$ lead to the same completion. This yields an open $S_0$-adic neighborhood of specializations $t\mapsto a$ which fulfill ii). 
Since an $S_0$-adic neighborhood is known not to be thin \cite[Proposition 3.5.3]{Ser}, it follows that there are infinitely many such specializations $t\mapsto a$ which also preserve the Galois group. 
\end{proof}

%


%
\subsection{General criteria}\label{sec:general}
The following criterion constructs, from a $K$-regular $G$-extension $N/K(t)$ with no universally ramified primes, infinitely many $G$-extensions with bounded ramification\footnote{If in addition $N/K(t)$ has a totally real specialization, the same proof shows that these $G$-extensions can also be chosen totally real.}, and hence infinitely many number fields of bounded degree over $K$ possessing unramified $G$-extensions.
%
%
\begin{thm}
\label{main_crit}
Let $N/K(t)$ be a $K$-regular Galois extension with group $G$, let $(C_1,...,C_r)$ be its ramification type, and $e_i$  the order of elements in $C_i$ ($i=1,...,r$). Assume that $U(N/K(t))=\emptyset$.
Then for every finite set $S$ of primes of $K$, there exist infinitely many $G$-extensions $F/K$  which are tamely ramified with all ramification indices dividing some $e_i$, $i=1,\ldots,r$, and unramified at $S$.
In particular, $e(K,G)\le \lcm\{e_1,...,e_r\}$.
\end{thm}
\begin{proof}
Enlarge $S$ to contain the set of exceptional primes for $N/K(t)$ in Theorem \ref{beckmann} and the set of prime divisors of $|G|$. 
By Prop.\ \ref{weak_appr}, there are infinitely many $a\in K$ such that $N_a/K$ has group $G$ and is unramified at all primes $p\in S$. Then Theorem \ref{beckmann} yields that all ramification indices in $N_a/K$ divide some $e_i$. Tame ramification follows since $S$ contains the primes dividing $|G|$. The last assertion follows from the first by Lemma \ref{lem:basic}. 
\end{proof}
\begin{remark}
\label{rem:main_crit}
The ramification indices $e_1,\ldots,e_r$ need not all occur in the specializations $N_a/K$; determining which ones occur leads to intriguing number-theoretical problems. However, an easy but useful special case should be noted: Assume that one of the branch points of $N/K(t)$, say the one corresponding to $C_r$, is $K$-rational. Via linear transformations in $t$, we may assume that it is $t\mapsto \infty$. Now the constructed extensions $N_a/K$ only require $a$ to be $p$-adically close to certain prescribed values for all $p\in S$. One can therefore without loss choose $a$ to be $q$-integral for all primes $q\notin S$ (i.e. $I_q(a,\infty):=I_q(1/a,0)\le 0$). So by Beckmann's theorem (and the assumption that $S$ contains all the bad primes of $N/K(t)$), all primes ramified in $N_a/K$ have ramification index dividing some $e_i$ with $i\in \{1,...,r-1\}$. 
\end{remark}
In  light of  Theorem \ref{main_crit}, we search for criteria which ensure $U(N/K(t))=\emptyset$.
Most obviously, this happens if 
the Galois extension $N/K(t)$ possesses a trivial specialization, i.e.\ a non-branch point $a\in K$ such that $N_a=K$ \footnote{This has an obvious generalization to the case where $N_a/K$ is unramified. We are however mostly interested in $K=\qq$, where the trivial extension is the only unramified one.
}. 
However,  without extending the base field the condition of a trivial specialization is usually difficult to ensure. 
%

The following theorem gives another scenario in which we can eliminate universally ramified primes. 
Its assertion, together with Theorem \ref{main_crit}, yields in particular that $e(K,G) \le \gexp(G)$ for the given number field $K$.
%

For $I\leq G$,  let $C_G(I)$ and $N_G(I)$ denote the centralizer and normalizer of  $I$ in $G$, respectively. Let $\mu_e$ denote the $e$-th roots of unity.
\begin{thm}
\label{coprime_in}
Let $N/K(t)$ be a $K$-regular $G$-extension with ramification type $(C_1,\ldots,C_r)$, let $I_i$ be the cyclic group generated by an element of  $C_i$ and $e_i$ its order, for $i=1,\ldots,r$. 
Assume $e_{1}$ and $e_2$ are the only indices not dividing $\gexp(G)$, assume $\gcd(e_1,e_2)=1$, 
and assume  $C_G(I_i)=I_i$, $i=1,2$.  \\
Then there exists a $K$-regular $G$-extension $N'/K(t)$ with all inertia groups of order dividing $\gexp(G)$ and with $U(N'/K(t))=\emptyset$.
\end{thm}
\begin{proof}
Let $t_1,t_2$ be the branch points corresponding to $C_1$ and $C_2$, respectively. 
Since by 
assumption, the classes $C_{1}$ and $C_2$ are both the only ones of their respective element order, the branch cycle argument (e.g. \cite{Voe}, 2.8) implies that $t_1,t_2$ are $K$-rational. 
After applying a fractional linear transformation in $t$, we can assume without loss of generality that $t_1,t_2$ are $0$ and $\infty$, respectively.

It is well known that the conjugation action of the decomposition group $D_i$ at $t_i$ on $I_i$ factors through the action of $\Gal(K(\mu_{e_i})/K)$ on $I_i\cong \mu_{e_i}$, $i=1,2$. Since the latter action is faithful and the kernel of the action of $D_i$ is $C_{D_i}(I_i) = I_{i}$, we get that the Galois group $D_i/I_i$ of the residue extension at $p_i$ is isomorphic to $\Gal(K(\mu_{e_i})/K)$. Since the residue extension $N_{t_i}$ at $t_i$ contains $\mu_{e_i}$, we deduce that $N_{t_i}=K(\mu_{e_i})$ for $i=1,2$. 

\vspace{2mm}

Since $e_1$ and $e_2$ are coprime by assumption, it follows that no finite prime of $K$ ramifies in both residue extensions $N_{t_i}/K$, $i=1,2$. We would ideally like to conclude from this the non-existence of universally ramified primes for $N/K(t)$. Note however that $t_1$ and $t_2$ are branch points. The further task is therefore to create a different regular $G$-extension, unramified at $0$ and $\infty$, whilst keeping control over the residue extensions at these points. This will be achieved by a suitable pullback, i.e., compositum of $N$ with a rational function field $K(v)\supset K(t)$. For the sake of transparency, we do this pullback in two separate steps, killing ramification at $0$ and at $\infty$ respectively.

\vspace{2mm}

{\textit{First step}}:
It follows from the previous that the completion $\hat N_{0}$  of $N/K(t)$ at $0$ is $K(\mu_{e_1})((\sqrt[e_1]{\alpha t}))$ for $\alpha\in K(\mu_{e_1})$.
Note that 
as  $\alpha$ is of integral $p$-adic valuation in $K(\mu_{e_1})$ for every $p$ dividing $e_2$
(since $p$ does not ramify in the cyclotomic extension $K(\mu_{e_1})/K$), there exists $\tilde \alpha\in K$ such that $\alpha/\tilde\alpha$ is of $p$-adic valuation $0$ for all $p$ dividing $e_2$. 
Then the extension $K(\mu_{e_1},\sqrt[e_1]{\alpha/\tilde\alpha})/K$ is unramified at all such $p$.
Take  a root $u$  of $X^{e_1}-\tilde\alpha t$ over $K(t)$, and consider the extension $N(u)/K(u)$. 
We claim that $N(u)/K(u)$ is a $K$-regular $G$-extension with residue extensions $K(\mu_{e_1},\sqrt[e_1]{\alpha/\tilde \alpha})/K$ and $K(\mu_{e_2})/K$ over $0$ and $\infty$ respectively.  To prove the claim, note that since $\infty$ has ramification index $e_2$ in $N$  and $e_1$ in $K(u)$, the extensions $N/K(t)$ and $K(u)/K(t)$ are linearly disjoint over $\overline K$, and hence $N(u)/K(u)$ is a $K$-regular $G$-extension. 
Moreover, since $e_1$ and $e_2$ are coprime, every place of $N$ over $\infty$ is totally ramified in $N(u)/N$, hence the residue field of $N(u)/K(u)$ at $\infty$ is $N_\infty=K(\mu_{e_2})$. 
By Abhyankar's lemma $N(u)/K(u)$ is unramified over $0$. Moreover, since $\hat N_0\cong K(\mu_{e_1})((\sqrt[e_1]{\alpha t}))$, the compositum  $\hat N_0\cdot K((u))$ is isomorphic to the field $$ K((\sqrt[e_1]{\tilde \alpha t}))(\mu_{e_1},\sqrt[e_1]{\alpha t}) = K(\mu_{e_1},\sqrt[e_1]{\alpha/\tilde\alpha})((u)).$$ Hence the residue field of $N(u)$ at $0$ is 
$K(\mu_{e_1},\sqrt[e_1]{\alpha/\tilde\alpha}),$ 
proving the claim. 

\vspace{2mm}

{\textit{Second step}}:
Let $\frak q$ be the product of the ramified primes in $K(\mu_{e_1},\sqrt[e_1]{\alpha/\tilde\alpha})/K$. As in the first step, the completion of $N(u)/K(u)$ at $\infty$ is $K(\mu_{e_2})((\sqrt[e_2]{\beta/u}))$ for some $\beta\in K(\mu_{e_2})$. Since $\frak q$ is coprime to $e_2$, there exists $\tilde\beta\in K$ with the same $p$-adic valuation as $\beta$ for all $p$ dividing $\frak q$. Let $v$ be a root of $X^{e_2}-\tilde\beta/u$ and consider the extension $N(v)/K(v)$. As in the above argument, the residue extensions of $N(v)/K(v)$ at $0$ and $\infty$ are $K(\mu_{e_1},\sqrt[e_1]{\alpha/\tilde\alpha})/K$ and $K(\mu_{e_2},\sqrt[e_2]{\beta/\tilde\beta})/K$, respectively. Since $0$ and $\infty$ are unramified in $N(v)$, since the primes $p$ not dividing $\frak q$ are unramified in $K(\mu_{e_1},\sqrt[e_1]{\alpha/\tilde\alpha})$, and since the primes dividing $\frak q$ are unramified in $K(\mu_{e_2},\sqrt[e_2]{\beta/\tilde\beta})$, the set $U(N(v)/K(v))$ is empty. 

Since $N(v)/K(v)$ is a $K$-regular $G$-extension with all inertia groups of order dividing $\gexp(G)$, 
the assertion follows.
\end{proof}

%
The advantage of the above theorem is that it does not require a concrete polynomial to determine the set of universally ramified primes, but rather depends only on group-theoretical data (the Galois group and ramification type of the extension). 

\subsection{Examples}

We exhibit several examples of almost simple groups $G$ which are realizable infinitely often over $\qq$ with all ramification indices dividing $\gexp(G)$. Note that the set of examples could easily be enlarged. Our selection intends to demonstrate the criteria exhibited in Theorem \ref{main_crit}, Remark \ref{rem:main_crit}, Theorem \ref{coprime_in}, and also includes the smallest sporadic simple group, the Mathieu group $M_{11}$. 
\begin{prop}
\label{ge2}
Let $G$ be $M_{11}$  or one of the groups $A_5,PSL_2(7), PSL_2(11)$, $PGL_2(7), PSL_3(3),  PSp_4(3).2, PSp_6(2)$. For every finite set $S\subset\mathbb{P}$ of prime numbers, there are infinitely many tamely ramified $G$-extensions of $\qq$ unramified at $S$ and with all inertia groups of order $2$. In particular, there exist infinitely many quadratic extensions $K/\qq$, unramified at $S$, such that $K$ possesses an everywhere unramified $G$-extension.
\end{prop}
\begin{proof}
We prove the first assertion, and the second follows from the first by applying Lemma \ref{lem:basic}. 
The  cases $PGL_2(7),   PSp_4(3).2, PSp_6(2)$ are consequences of the criterion in Theorem \ref{coprime_in}. We verify that the theorem applies  purely group-theoretically using Magma:
The group  $PGL_2(7)$ has a rationally rigid triple $C$ of conjugacy classes of element orders $2$, $6$, and $7$ respectively.\footnote{Recall that a conjugacy class is called rational, if its elements $x$ are conjugate to all powers coprime to the element order, or equivalently, if $N_G(\langle x\rangle)$ projects onto $Aut(\langle x\rangle)$ via its conjugation action on $\langle x\rangle$. A class tuple is then called rational if all occurring classes are rational. }  Thus by the rigidity criterion, cf.\ e.g.\ \cite[Theorem I.4.8]{MM}, there exists a $G$-extension $N/\mQ(t)$ with ramification type $C$. 
The normalizer of an element of order $6$ in $PGL_2(7)$ is the dihedral group $D_6$ of order $12$, and for order $7$, it is $\zz/7\zz\rtimes \zz/6\zz$.  
 Therefore the assumptions of Theorem \ref{coprime_in} are fulfilled. 
Similarly, the group $PSp_4(3).2$ has a rationally rigid triple of classes of element orders $(2,8,9)$, and the classes of elements of order $8$ and $9$ respectively fulfill the required assumptions.
Finally, the simple group $PSp_6(2)$ has a rationally rigid triple of classes of element orders $(2,7,9)$,
and the classes of elements of order $7$ and $9$ respectively fulfill the required assumptions. 

In the cases $G=M_{11}$, $PSL_2(11)$ or $A_5$, we claim that there exists a regular $G$-extension with all but one ramification index equal to $2$, and  no universally ramified primes. Applying Theorem \ref{main_crit} together with Remark \ref{rem:main_crit} then yields  $G$-extensions of $\qq$ with all inertia groups of order $2$.

The polynomial $f(t,X)$ given in \cite[Thm.\ 1]{K_m11} has regular Galois group $M_{11}$ over $\qq(t)$ and ramification type $(2A,2A,3A,5A)$ (with the last two branch points equal to $\infty$ and $0$). 
Take $N/\qq(u)$ to be the splitting field of $f(2u^3,X)$. As in the proof of Theorem \ref{coprime_in}, this is a translation of a $G$-extension of $\mQ(t)$ where $t=2u^3$, by the extension $\mQ(u)/\mQ(t)$ which is totally ramified, and hence $N/\qq(u)$ is a $\qq$-regular $M_{11}$-extension. By Abhyankar's lemma its  ramification type is therefore $(2A,2A,2A,2A,2A,2A,5A)$.
To verify that there are no universally ramified primes, note that the root field of $f(2\cdot (9/5)^3,X)$ has prime power discriminant $45513961^4$, and $45513961$ is not a universally ramified prime either, as it is unramified in the splitting field of $f(2,X)$.
For $G=PSL_2(11)$, the appendix of \cite{MM} gives a $G$-extension, given by the polynomial 
$$\begin{array}{rl} f(t,X) = & X^{11} - 3X^{10} + 7X^9 - 25X^8 + 46X^7 - 36X^6 + 60X^4 \\ & -121X^3+140X^2-95X+27+tX^2(X-1)^3.\end{array}$$ 
Hence the corresponding cover $\mathbb P^1\ra\mathbb P^1$ is given by the rational function 
$$x\mapsto \frac{x^{11} - 3x^{10} + 7x^9 - 25x^8 + 46x^7 - 36x^6 + 60x^4 -121x^3+140x^2-95x+27}{x^2(x-1)^3}.$$ Its ramification is easily computed by finding the order of the roots of the derivatives, and is given by the tuple $(2A,2A,2A,6A)$. 
Specializing $t\mapsto 1$ and $t\mapsto 2$, gives residue extensions with coprime discriminant. 
For $G=A_5$, a regular $G$-extension with ramification type $(2A,2A,2A,3A)$ and without universally ramified primes is given (via a polynomial $f(0,v,x)$) in \cite[Theorem 3.1]{KRS}.

 In the case $G=PSL_2(7)$,  an extension of $\qq(t)$ with all inertia groups of order $2$ and without universally ramified primes is deduced from \cite[Proof of Theorem 3.2]{KRS} by specializing some of the parameters. 
Finally, a $PSL_3(3)$-extension of $\qq(t)$ with all inertia groups of order $2$ is given in \cite[Lemma 3.4]{KRS}. Specializing that polynomial at $t\mapsto 1$ and $t\mapsto 3$ gives extensions with coprime discriminant. 

\end{proof}

\begin{remark}
\label{rem:an}
As can be seen in the proof of Proposition \ref{ge2}, successful application of Theorem \ref{coprime_in} for a group $G$ requires two ingredients: first, the existence of two rational conjugacy classes $C_1$ and $C_2$ with trivial centralizers and coprime element orders; and secondly, a regular Galois realization with group $G$ (e.g., using the rigidity criterion or a braid genus criterion as in \cite[Theorem III.7.8]{MM}) whose ramification type contains only classes of involutions apart from $C_1$ and $C_2$.
It is worth pointing out that the first condition is fulfilled for all alternating groups $A_n$, with a few small exceptions. In particular, for $n\ge 7$ odd, the classes $C_1$ of cycle structure $(n-2.1.1)$ and $C_2$ of cycle structure $(n-3.2.1)$ are easily verified to fulfill the assumptions. To apply Theorem \ref{coprime_in}, it would then suffice to show the existence of a rational point on any one of infinitely many Hurwitz spaces of $A_n$-covers with ramification type $(C_1,...,C_r)$, where $C_3$,..., $C_r$ are classes of involutions in $A_n$ and $r\ge 3$ is arbitrary. Such existence results for all $n$ have at least been obtained in certain comparable situations, cf.\ e.g.\ \cite{Hallouin}. This gives hope that $e(\qq,A_n) = 2$ can be shown in general with our methods.
\end{remark}

\subsection{Basic operations}\label{sec:basic}
We generate further classes of examples via suitable group extensions.
We first remark that the above constructions are closed under finite direct products. 
\begin{remark}\label{rem:direct}Let $G$, $H$ be finite groups, and $K$ be a number field. 
Assume  there exists a tame $H$-extension $L_1/K$ with ramification indices dividing $\gexp(H)$ and infinitely many pairwise linearly disjoint tame $G$-extensions with ramificiation dividing $\gexp(G)$. Then by picking one of the above $G$-extensions $L_2/K$ which is linearly disjoint from $L_1$, the extension $L_1L_2/K$ has Galois group $G\times H$ and its inertia groups are of order dividing $\lcm(\gexp(G),\gexp(H))$. 
Since $\gexp(G\times H) = \lcm(\gexp(G),\gexp(H))$ this gives $e(\qq,G\times H)= \gexp(G\times H)$.
\end{remark}

Next we show that for groups $G$ with $\gexp(G)=2$ as in Proposition \ref{ge2}, 
one can realize any iterated wreath product $((G \wr S_{n_1})\wr \cdots \wr S_{n_{r-1}}) \wr S_{n_r}$ 
with ramification of order $2$. Here $H\wr S_n:=H^{n}\rtimes S_n$ is the wreath product with respect to the natural action of $S_n$ on $\{1,\ldots,n\}$. 
\begin{prop}\label{prop:wreath}
Let $K$ be a number field, $G$ be a finite group and $n\in \mathbb{N}$. Suppose $G$ is realizable over $K(t)$ with at most one branch point of ramification index $>2$, and with no universally ramified primes. Then the same is true for $G \wr S_n$. In particular, $e(\qq, G\wr S_n)=2$.
\end{prop}
\begin{proof} 
Let $N/K(t)$  be such a $G$-extension and denote by $B$ its set of branch points. 
By Hilbert's irreducibility theorem there exists a tuple $ \alpha = (\alpha_1,\ldots,\alpha_n)\in \mathbb A_K^n$ of points outside  $B$, such that the residue extensions $N_{\alpha_i}/K$, $i=1,...,n$, are linearly disjoint $G$-extensions. 
For a tuple $ b=(b_1,\ldots,b_n)$ of variables, consider the polynomial
$$ f_{ \alpha, b}(u,T) = (T-\alpha_1)\cdots (T-\alpha_n) - u(T-b_1)\cdots(T-b_n) \in K[b_1,\ldots,b_n,u,T].$$
For each value $\beta=(\beta_1,\ldots,\beta_n)\in \mathbb A_K^n$ of $ b$, let $u_\beta=\frac{\prod_{i=1}^n(t-\alpha_i)}{\prod_{i=1}^n(t-\beta_i)}\in K(t)$. By Gauss's lemma, $f_{\alpha,\beta}(u_\beta,T)$ is the minimal polynomial for $t$ over $K(u_\beta)$.

We claim that there are infinitely many values $\beta=(\beta_1,\ldots,\beta_n)\in (\mathbb A_K\setminus B)^n$ for $ b$ such that:
\begin{enumerate}
\item all inertia groups of the splitting field $\Omega_1$ of $f_{ \alpha, \beta}(u_\beta,T)\in K(u_\beta)[T]$ are generated by transpositions;
\item no branch point of $N/K(t)$ lies over a branch point of $K(t)/K(u_\beta)$;
\item the set of ramified primes in $N_{\beta_j}/K$ is disjoint from that of $N_{\alpha_i}/K$ for all $i,j\in\{1,\ldots,n\}$. 
\end{enumerate}

The remaining proof is then divided into two parts.

\vspace{2mm}

{\textit{First part: Assuming the claim, we prove the theorem}}.   Fix $\beta$ fulfilling (1), (2) and (3), and let $\Omega$ be the Galois closure of $N/K(u_\beta)$. 
We show  that $\Omega/K(u_\beta)$ is a Galois extension with all except at most one ramification index of order $2$, with no universally ramified primes, and such that $\Gal(\Omega/K(u_\beta))\cong G \wr S_n$. The theorem then follows from Theorem \ref{main_crit} and Remark \ref{rem:main_crit}. 
 
To show that all except at most one of the ramification indices are $2$, note that the branch points of $\Omega/K(u_\beta)$ are either branch points of $K(t)/K(u_\beta)$ or the restrictions of branch points of $N/K(t)$ to $K(u_\beta)$. By conditions (1) and (2), every branch point of $K(t)/K(u_\beta)$ is of ramification index $2$ and is unramified in $N/K(t)$. Thus, as $\Omega$ is the compositum of conjugates of $N/K(u_\beta)$, Abhyankar's lemma implies that each branch point of $K(t)/K(u_\beta)$ has ramification index $2$ in $\Omega$. 
By assumption all except at most one branch point of $N/K(t)$ have inertia group of order $2$, and by condition (2) they are unramified in  $K(t)/K(u_\beta)$. Thus the restriction of any such place to $K(u_\beta)$ has ramification index $2$ in $N/K(u_\beta)$, and in $\Omega/K(u_\beta)$ as well, by Abhyankar's lemma.

To show that there are no universally ramified primes, note that since the places  $t\mapsto \alpha_i$, $i=1,\ldots,n$ are unramified in $N/K(t)$ and their restriction to $K(u_\beta)$ is $u_\beta\mapsto 0$, the residue extension of $\Omega/K(u_\beta)$ over $u_\beta\mapsto 0$ is the compositum of $N_{\alpha_i}$, $i=1,\ldots,n$. Similarly, its residue extension over $\infty$ is the compositum of $N_{\beta_j}$, $j=1,\ldots,n$. By condition (3), the sets of primes ramifying in the two composita are disjoint, showing $U(\Omega/K(u_\beta))=\emptyset$. 

Finally, since the group $\Gal(\Omega_1/K(u_\beta))$ is a transitive subgroup of $S_n$ which is generated by transpositions by condition (1), it is isomorphic to $S_n$. 
We deduce a natural embedding of the Galois group $\Gal(\Omega/K(u_\beta))$ into $G \wr S_n$, projecting onto $S_n$. By the choice of $\alpha_1,...,\alpha_n$, the group $\Gal(\Omega/K(u_\beta))$ has a subgroup isomorphic to the kernel $G^n$ of that projection, namely the decomposition group at $u_\beta\mapsto 0$. This shows $\Gal(\Omega/K(u_\beta)) = G\wr S_n$.

\vspace{2mm}

{\textit{Second part: proving the claim}}.  We show that the set of $\beta\in(\mathbb A_K\setminus B)^n$ satisfying  the first and second condition is Zariski open, and those satisfying the third are $R$-adically open, where $R$ is the set of all primes ramified in at least one of the residue extensions $N_{\alpha_i}/K$, $i=1,\ldots,n$. Since an $R$-adically open subset is Zariski dense, its intersection with the two non-empty Zariski open subsets is Zariski dense and hence infinite, which would prove the claim.

For the first condition, note that the discriminant $\Delta_{ \alpha, b}(u)$ of $f_{\alpha,b}(u,T)$ with respect to $T$ is a polynomial in $u$ whose coefficients are rational functions in the $b_i$'s. It is even square free \footnote{Indeed, for all $u\mapsto u_0\in K\setminus\{0\}$, the polynomial $f_{\alpha,b}(u_0,T)$, viewed over the ring of symmetric functions of the $b_i$, is a generic degree-$n$ polynomial, and therefore well-known to have square free discriminant.}, whence the discriminant $\Delta$  of the polynomial $\Delta_{ \alpha, b}(u)$ with respect to $u$ is a nonzero rational function in $b$. Thus the condition $\Delta\neq 0$ is a Zariski open condition on the set of $\beta\in (\mathbb A_K\setminus B)^n$.
In other words, $f_{\alpha,\beta}(u,T)$ has square free discriminant for a Zariski open set of values $\beta$. But a square free polynomial discriminant implies a square free relative discriminant of the extension $K(t)/K(u_\beta)$, meaning that all inertia groups are generated by transpositions, as desired. 

For the second condition, note that the place $t \mapsto t_i$ of $K(t)$ extends the place $u_\beta\mapsto u_\beta(t_i)$ in $K(u_\beta)$. Therefore, to ensure that $t_i\in B$ does not lie over a branch point of $K(t)/K(u_\beta)$, it is sufficient to have $\Delta_{ \alpha, \beta}(u_{\beta}(t_i))\neq 0$. 
This condition is Zariski open, as desired. 

For the third condition, 
note that as $U(N/K(t))=\emptyset$,   Proposition \ref{weak_appr} (and its proof) with $S$ taken to be $R$ gives an $R$-adically open set of values $\beta$   such that $N_{\beta_i}$ is unramified in $R$ for all $i=1,\dots,n$, proving the claim. 
\end{proof}

\end{document}